\documentclass[a4paper,reqno]{amsart}

\usepackage{enumerate}
\usepackage[colorlinks=true]{hyperref}
\hypersetup{urlcolor=blue, citecolor=cyan}
\hypersetup{citecolor=blue}

\numberwithin{equation}{section}

\newtheorem{thm}{Theorem}[section]
\newtheorem{lem}[thm]{Lemma}

\newtheorem{prop}[thm]{Proposition}
\theoremstyle{definition}
\newtheorem{rem}[thm]{Remark}

\newenvironment{notation}{\medskip \noindent{\bf Notation. }}{}

\newcommand\R{{\mathbb R}}
\newcommand\C{{\mathbb C}}
\newcommand\N{{\mathbb N}}

\newcommand\Tma{T_{\mathrm{max}}}

\newcommand\Ens{{\mathcal E}}
\newcommand\Spa{{\mathcal X}}
\newcommand\Spb{{\mathcal Y}}

\newcommand\dist{{\mathrm{d}}}

\newcommand\Srn{{\mathcal S}(\R^N )}

\newcommand\Sma{S_{\mathrm{max}}}
\newcommand\Sig{{\Sigma}}

\newcommand\Inzu{m}
\newcommand\Inzd{k}
\newcommand\Inzt{n}
\newcommand\Inzq{\ell}
\newcommand\Inzc{j}

\newcommand\Intd{k}

\newcommand\Intq{\ell}
\newcommand\Intc{j}
\newcommand\Ints{\mu }

\newcommand\Inth{\nu }

\newcommand\Imtq{{\ell}}

\newcommand\Imqu{{m}}
\newcommand\Imqd{{n}}
\newcommand\Imqt{{k}}
\newcommand\Imqq{{\ell}}
\newcommand\Imqc{{j}}
\newcommand\Imqs{{\mu}}
\newcommand\Imqp{{\nu}}
\newcommand\Imqh{{J}}

\newcommand\Loc{{\mathrm{loc}}}

\newcommand\goto{\mathop{\longrightarrow}}

\newcommand\MScN[1]{\href{http://www.ams.org/mathscinet-getitem?mr=#1}{\nolinkurl{(#1)}}}
\newcommand\DOI[1]{\href{http://dx.doi.org/#1}{(doi: \nolinkurl{#1})}}
\newcommand\LINK[1]{\href{#1}{(link: \nolinkurl{#1})}}

\newcommand\DI{u_0 }
\newcommand\DIb{v_0 }

\begin{document}

\title{Local existence, global existence, and scattering for the nonlinear Schr\"o\-din\-ger equation}

\def\shorttitle{nonlinear Schr\"o\-din\-ger equation}

\author[T. Cazenave]{Thierry Cazenave}
\address{Universit\'e Pierre et Marie Curie \& CNRS, Laboratoire Jacques-Louis Lions,
B.C. 187, 4 place Jussieu, 75252 Paris Cedex 05, France}
\email{\href{mailto:thierry.cazenave@upmc.fr}{thierry.cazenave@upmc.fr}}

\author[I. Naumkin]{Ivan Naumkin}
\address{Universit\'e Pierre et Marie Curie, Laboratoire Jacques-Louis Lions,
B.C. 187, 4 place Jussieu, 75252 Paris Cedex 05, France}
\email{\href{mailto:ivan.naumkin@upmc.fr}{ivan.naumkin@upmc.fr}}



\begin{abstract}

In this paper, we construct for every $\alpha >0$ and $\lambda \in \C$ a space of initial values for which there exists a local solution of
the nonlinear Schr\"o\-din\-ger equation
\begin{equation*} 
\begin{cases} 
iu_t + \Delta u + \lambda  |u|^\alpha u= 0 \\  u(0,x) =   \DI
\end{cases} 
\end{equation*} 
on $\R^N $. Moreover,  we construct for every $\alpha >\frac {2} {N}$ a class of (arbitrarily large) initial values for which there exists a global solution that scatters as $t\to \infty $.

\end{abstract}

\maketitle

\section{Introduction} \label{sIntro} 

In this article, we study the existence of local and global solutions of the nonlinear Schr\"o\-din\-ger equation 
\begin{equation} \label{NLS1} 
\begin{cases} 
iu_t + \Delta u + \lambda  |u|^\alpha u= 0 \\  u(0,x) =   \DI
\end{cases} 
\end{equation} 
on $\R^N $, where $\alpha >0$ and $\lambda \in \C$, or its  equivalent formulation
\begin{equation} \label{NLS1:i} 
u(t) = e^{it\Delta } \DI + i \lambda \int _0^t e^{i (t-s) \Delta } |u(s)|^\alpha u(s)\, ds
\end{equation} 
where $(e^{it\Delta }) _{ t\in \R }$ is the Schr\"o\-din\-ger group. 

Concerning the local theory, the relevant space in which to study the Cauchy problem appears to be the Sobolev space $H^s (\R^N ) $. Local well-posedness is well-known in $L^2$ if $\alpha <\frac {4} {N}$ (see~\cite{Tsutsumi1}), in $H^1$ is $\alpha <\frac {4} {N-2}$ (see~\cite{GinibreV1}), and in $H^2$ if $\alpha <\frac {4} {N-4}$ (see~\cite{Kato1}). 
More generally, the problem is locally well-posed in $H^s$ if $0\le s< \frac {N} {2}$ and $\alpha <\frac {4} {N-2s}$, but under the additional condition $\alpha > [s]$ if $s> 1$ and $\alpha $ is not an even integer. 
(Here,  $[s]$ the integer part of $s$)
This condition appears because the map $u\mapsto  |u|^\alpha u$ is $C^\infty $ if $\alpha $ is an even integer; it is $C^{\alpha }$, but not $C^{\alpha +1}$ if $\alpha $ is an odd integer; and it is $C^{[\alpha ]+1}$ but not $C^{[\alpha ]+2}$ if $\alpha $ is not an integer. 
It appears naturally. Indeed, solutions are constructed by a fixed-point argument, for which one is lead to estimate derivatives of order up to $s$ of $ |u|^\alpha u$. 
When $s>\frac {N} {2}$ is an integer, local existence in $H^s$ is proved in~\cite{GinibreV2} under the same assumption: $\alpha > [s]$ if $\alpha $ is not an even integer. 
This condition was improved in certain cases, see~\cite{Pecher, FangH}, but not eliminated except for $s\le 2$.
For instance, it seems that no available local theory applies to the case $N=12$ and $\alpha =1$, and that
 there is no $\DI\not = 0$ for which the existence of a local solution (in some sense) of~\eqref{NLS1} is known\footnote{This observation is for the case of a general $\lambda $. If $\lambda \in \R$ and $\lambda <0$, then the existence of a (global) weak solution for $\DI \in H^1 (\R^N ) \cap L^{\alpha +2} (\R^N ) $ follows from compactness arguments, see~\cite{Strauss2, Strauss3}.}. There is some evidence that such a regularity assumption is not purely technical, see~\cite{CazenaveDW-Reg}. 

A regularity condition also appears for the low-energy scattering problem. It is a natural conjecture that if $\alpha >\frac {2} {N}$, then small initial values (in an appropriate sense) give rise to global solutions of~\eqref{NLS1} that are asymptotically free,  i.e. behave like a solution of the linear equation  as $t\to \infty $. This property is known in dimension $N=1,2,3$, see~\cite{Strauss1, CWrapdec, GinibreOV, NakanishiO}.
However, in larger dimension, the available methods leave a gap. This gap is not only due to the limitations discussed above, but also concerns values of $\alpha $ close to $\frac {2} {N}$, for which local existence is not an issue. The difficulty that arises clearly appears by using the pseudo-conformal transformation through which, a global, asymptotically free solution of~\eqref{NLS1} corresponds (see Section~\ref{sNLS}) to a solution of the nonautonomous equation 
\begin{equation} \label{NLS2} 
\begin{cases} 
iv_t + \Delta v + \lambda (1- b t)^{-\frac {4 - N\alpha } {2}} |v|^\alpha v= 0 \\  v(0,x) =  \DIb (x)
\end{cases} 
\end{equation} 
which has a limit (in an appropriate space) as $t\to \frac {1} {b}$. (Here, $b>0$ is a constant.)
The problem is then to solve~\eqref{NLS2} on $[0,\frac {1} {b})$. Note that the assumption $\alpha >\frac {2} {N}$ implies that $(1- b t)^{-\frac {4 - N\alpha } {2}}$ is integrable at $\frac {1} {b}$. 
However, the singularity at $\frac {1} {b}$ makes it problematic to apply Strichartz's estimates when $\alpha $ is close to $\frac {2} {N}$, see~\cite{CWrapdec, CazenaveCDW-Fuj}.
One can try another approach, and use the integrability of $(1- b t)^{-\frac {4 - N\alpha } {2}}$ by estimating the  $L^\infty $-norm of the solution, rather than applying Strichartz's estimates. However, the only way of controlling the $L^\infty $-norm seems to be by a control of the $H^s$-norm, for $s>\frac {N} {2}$, via Sobolev's inequality. 
When $N$ is large, we again face the problem of lack of regularity of the nonlinearity.

In this paper, we construct for every $\alpha >0$ a class of initial values for which there exists a local solution of~\eqref{NLS1}. Moreover,  we construct for every $\alpha >\frac {2} {N}$ a class of initial values for which there exists a global solution of~\eqref{NLS1} that scatters.
Before stating our results, we introduce some notation. 
We fix $\alpha >0$, we consider three integers $\Imqt ,\Imqu , \Imqd $ such that
\begin{equation} \label{fDInt1} 
\Imqt > \frac {N} {2}, \quad \Imqd > \max \Bigl\{ \frac {N} {2} +1,  \frac {N} {2\alpha } \Bigr\}, \quad  2 \Imqu \ge \Imqt + \Imqd +1
\end{equation} 
and we let
\begin{equation} \label{fSpa1b1} 
\Imqh = 2\Imqu +2 + \Imqt+ \Imqd .
\end{equation} 
We define the space $\Spa$ by
\begin{equation} \label{fSpa1} 
\begin{split} 
\Spa=  \{ u\in H^\Imqh  (\R^N ); & \, \langle x\rangle ^\Imqd D^\beta u \in L^\infty  (\R^N )  \text{ for  }  0\le  |\beta |\le 2\Imqu  \\   \langle x\rangle ^\Imqd D^\beta u  \in L^2 & (\R^N )  \text{ for  }   2\Imqu +1 \le  |\beta | \le 2\Imqu +2 + \Imqt,  \\  \langle x\rangle ^{\Imqh -  |\beta |} D^\beta u & \in L^2  (\R^N )  \text{ for  }    2\Imqu +2 + \Imqt <  |\beta | \le J \}
\end{split} 
\end{equation} 
and we equip $\Spa$ with the norm
\begin{equation} \label{fSpa2}
 \| u \|_\Spa =  \sum_{ \Imqc =0 }^{2\Imqu  }   \sup _{  |\beta |=    \Imqc }  \| \langle x\rangle ^\Imqd D^\beta  u  \| _{ L^\infty  } +   \sum_{\Imqp =0 }^{\Imqt +1} \sum_{ \Imqs =0 }^\Imqd \sum_{  |\beta  |=\Imqp + \Imqs  +  2 \Imqu  +1  }   \| \langle x \rangle ^{\Imqd -\Imqs } D^\beta  u \| _{ L^2 } 
\end{equation} 
where
\begin{equation}  \label{fSob2} 
\langle x \rangle = (1 +  |x|^2 )^{\frac {1} {2}}.
\end{equation}  

\begin{rem} \label{eRem1} 
Here are some comments on the space $\Spa$ defined by~\eqref{fSpa1}-\eqref{fSpa2}.  
\begin{enumerate}[{\rm (i)}] 

\item \label{eRem1:1} It follows from standard considerations that $(\Spa,  \| \cdot  \|_\Spa )$ is a Banach space.

\item \label{eRem1:2} Note that $ 2\Imqd > N+2$ by~\eqref{fDInt1}, so that $ \| \langle x\rangle w \| _{ L^2 } \le C  \| \langle x\rangle ^\Imqd  w  \| _{ L^\infty  } $. It easily follows that
\begin{equation} \label{fSpa1:b2} 
\Spa \hookrightarrow H^\Imqh (\R^N ) \cap \Sig
\end{equation} 
where
\begin{equation}  \label{fSpa1:b3} 
\Sig = H^1 (\R^N ) \cap L^2 (\R^N ,  |x|^2 dx) .
\end{equation}

\item \label{eRem1:3} It is immediate that $\Srn \subset \Spa$. 
Furthermore, it is not difficult to show that $\langle x\rangle ^{- p} \in \Spa$ if $p\ge \Imqd$
(apply~\eqref{fSob3}).
\end{enumerate}   
\end{rem} 

Our main results are the following.

\begin{thm} \label{eThm1b1} 
Let $\alpha >0$ and $\lambda \in \C$. Assume~\eqref{fDInt1}-\eqref{fSpa1b1} and let $\Spa$ be defined by~\eqref{fSpa1}-\eqref{fSpa2}.  
If  $\DI \in \Spa$  satisfies
\begin{equation} \label{eThm1:1} 
\inf  _{ x\in \R^N  }  \langle  x \rangle ^\Imqd  |\DI  (x) | >0 
\end{equation} 
then there exist $T>0$ and a unique solution $u\in C([-T,T], \Spa )$  of~\eqref{NLS1:i}.  
\end{thm} 

\begin{thm} \label{eThm2} 
Let $\alpha >\frac {2} {N}$ and $\lambda \in \C$. Assume~\eqref{fDInt1}-\eqref{fSpa1b1}, let $\Spa$ be defined by~\eqref{fSpa1}-\eqref{fSpa2} and $\Sig$ by~\eqref{fSpa1:b3}.  
Suppose $\DI = e^{i\frac {b |x|^2} {4}} \DIb $, where $b\in \R$ and $\DIb \in \Spa$ satisfies~\eqref{eThm1:1}. 
If $b>0$ is sufficiently large, then there exists a unique, global solution $u\in C([0,\infty ), \Sig ) \cap L^\infty ((0,\infty ) \times \R^N )$  of~\eqref{NLS1:i}.  Moreover $u$ scatters, i.e. there exists $u^+\in \Sig$ such that $e^{-it\Delta } u(t) \to u^+$ in $\Sig$ as $t\to \infty $. In addition, $\sup  _{ t\ge 0 } (1+t)^{\frac {N} {2}}  \| u(t) \| _{ L^\infty  }<\infty $.
\end{thm} 

\begin{rem}  \label{eRem2} 
Here are some comments on Theorems~\ref{eThm1b1} and~\ref{eThm2}.
\begin{enumerate}[{\rm (i)}] 

\item \label{eRem2:0} 
We will prove a slightly more general version of Theorem~\ref{eThm1b1}, in which the admissible initial values are not only the functions of $\Spa$ that satisfy~\eqref{eThm1:1}, but also all the functions that are obtained by multiplying them by $e^{i\frac {b |x|^2} {4}}$, where $b$ is any real number. 
This is Theorem~\ref{eThm1} below, from which Theorem~\ref{eThm1b1} follows immediately by choosing $b=0$. 

\item \label{eRem2:1} We note that if $u\in C([a,b], \Sig ) \cap L^\infty ((a,b) \times \R^N )$, then $ |u|^\alpha u \in C([a,b], \Sig )$. Therefore, equation~\eqref{NLS1:i} makes sense in $ \Sig $ for $u$ as in Theorems~\ref{eThm1b1}, \ref{eThm2}, and~\ref{eThm1}.

\item \label{eRem2:2} Note that we have the choice on the parameters $\Imqt ,\Imqu , \Imqd $ as long as they satisfy~\eqref{fDInt1}.  In particular, $\Imqd $ can be any integer satisfying the second condition in~\eqref{fDInt1}.

\item \label{eRem2:3} It follows from Remark~\ref{eRem1}~\eqref{eRem1:3} that Theorem~\ref{eThm2} applies to the initial value $\DI = z e^{i\frac {b |x|^2} {4}}  ( \langle x\rangle ^{- \Imqd} + \psi  )$ where $ \Imqd > \max \{ \frac {N} {2} +1,  \frac {N} {2\alpha } \} $, $\psi \in \Srn$ satisfies $ \|  \langle x\rangle ^{ \Imqd} \psi  \| _{ L^\infty  } <1$, $z\in \C$ and $b>0$  is sufficiently large. Theorem~\ref{eThm1b1} applies if $b=0$, and Theorem~\ref{eThm1} applies if $b$ is any real number.

\item \label{eRem2:4}
The solution constructed given by Theorem~\ref{eThm2} has  stronger regularity properties than stated. See Remark~\ref{eRem3}.

\item \label{eRem2:5}
Note that there are no restrictions on the size of the initial value in Theorems~\ref{eThm1b1} and~\ref{eThm2}. Besides the smoothness and decay imposed by the assumption $\DI \in \Spa$ (or $\DIb \in \Spa$), the only limitation is condition~\eqref{eThm1:1}. Note that if $\DI\in \Spa$ satisfies~\eqref{eThm1:1}, then
\begin{equation*} 
0< \liminf  _{  |x|\to \infty }  |x|^\Imqd  |\DI (x) | \le \limsup  _{  |x|\to \infty }  |x|^\Imqd  |\DI (x) | <\infty .
\end{equation*} 

\item \label{eRem2:6}
The condition $\alpha >\frac {2} {N}$ in Theorem~\ref{eThm2} cannot be replaced by $\alpha >  \underline{\alpha } $ for some $ \underline{\alpha } <\frac {2} {N}$.   Indeed, if $\alpha <\frac {2} {N}$ and $\Im \lambda <0$, then it follows from \cite[Theorem~1.1]{CazenaveCDW-Fuj} that every $H^1$-solution of~\eqref{NLS1} blows up in finite or infinite time. Thus we see that no nontrivial solution of~\eqref{NLS1} can satisfy the conclusion of Theorem~\ref{eThm2}. 
Moreover, if $\alpha <\frac {2} {N}$ and $\Im \lambda \ge 0$, then all $L^2$ solutions are global, but they do not scatter.  (See Strauss~\cite{Strauss4}, Theorem~3.2 and Example~3.3, p.~68. See also \cite{Barab} for the one-dimensional case.) In particular, the case $\alpha =\frac {2} {N}$ is critical. 

\item \label{eRem2:7}
In the range $\alpha _0 < \alpha <\frac {4} {N-2}$, where $\alpha _0$ is the positive zero of $N\alpha ^2 + (N-2) \alpha = 4$, the conclusion of Theorem~\ref{eThm2} (except for the $L^\infty $ decay estimate) follows from~\cite[Corollary~2.5]{CWrapdec}. Note that the assumptions in~\cite[Corollary~2.5]{CWrapdec} concerning $\DIb$ are less restrictive than in  Theorem~\ref{eThm2}, it is only required that $\DIb\in \Sig$.

\item \label{eRem2:8}
Theorem~\ref{eThm2} does not say anything on what happens to the solution $u$ for $t<0$. In fact, one cannot in general expect that the initial values considered in Theorem~\ref{eThm2} give rise to global solutions for negative times. Indeed, suppose $\frac {4} {N}\le \alpha <\frac {4} {N-2}$ (so that the Cauchy problem~\eqref{NLS1:i} is locally well-posed in $\Sig$) and that $\lambda \in \R$, $\lambda >0$. Let $\DIb$ satisfy the assumptions of Theorem~\ref{eThm2} and suppose further that $\frac {1} {2}  \| \nabla \DIb\| _{ L^2 }^2 - \frac {1} {\alpha +2}  \| \DIb\| _{ L^{\alpha +2} }^{\alpha +2} <0$. (This can be achieved by multiplying $\DIb$ by a sufficiently large constant.)
Let $u$ be the corresponding solution of~\eqref{NLS1:i} with $\DI = e^{i \frac {b |x|^2} {4}} \DIb $ defined on the maximal interval $(-\Sma, \Tma)$.  
It follows from Theorem~\ref{eThm2} (or~\cite[Corollary~2.5]{CWrapdec}) that if $b>0$ is sufficiently large, then $\Tma=\infty $ and $u$ scatters. 
On the other hand, it follows from \cite[Remark~2.6]{CWrapdec} that for every $b>0$, $\Sma <\infty $.

\item  \label{eRem2:9}
We can apply Theorem~\ref{eThm2} to construct solutions of~\eqref{NLS1:i} that exist for all $t<0$ and scatter as $t\to -\infty $. Indeed, it suffices to apply Theorem~\ref{eThm2} to equation~\eqref{NLS1:i} with $\lambda $ replaced by $ \overline{\lambda } $. If $\DI$ satisfies the assumptions of Theorem~\ref{eThm2} (for $ \overline{\lambda } $) and $u$ is the corresponding solution, then we see that $v(t) =  \overline{u}  (-t)$ is a solution of~\eqref{NLS1:i} (with $\lambda $) for $t<0$, which scatters as $t\to -\infty $, and with initial value $ \overline{\DI} $. 
Of course, one cannot expect in general that $v$ is global for positive times, since this would mean that $u$ is global for negative times. (See~\eqref{eRem2:8} above.)

\end{enumerate} 
\end{rem} 

Our strategy for proving Theorem~\ref{eThm1b1} is based on the following observation:
Since the possible defect of smoothness of the nonlinearity $ |u|^\alpha u$ is only at $u=0$, there is no obstruction to regularity for a solution that does not vanish. This suggests to look for such solutions. This is not completely trivial, since there is no maximum principle for the Schr\"o\-din\-ger equation, and this is why the various conditions in the definition of the space $\Spa$ arise. 
Indeed, consider for instance $\psi  (x)= \langle x\rangle ^{-\Imqd}$, where $\Imqd >\frac {N} {2}+1$ so that $\psi \in \Sig$, and let $v(t)= e^{it\Delta }\DI$ be the solution of
\begin{equation} \label{Low2} 
\begin{cases} 
iv_t + \Delta v=0 \\
v (0, x)=  \psi (x).
\end{cases} 
\end{equation} 
We want to estimate $\inf  _{ x\in  \R^N }  \langle x\rangle ^\Imqd  | v (t, x) | $
and we note that 
\begin{equation} \label{fInt1} 
v(t,x)= \psi (x) + i\int _0^t \Delta v(s,x) \, ds .
\end{equation} 
Therefore,
\begin{equation*} 
\langle x\rangle ^\Imqd  | v(t,x) | \ge \langle x\rangle ^\Imqd   |\psi  (x)| - \int _0^t \langle x\rangle ^\Imqd  |\Delta v|
\end{equation*} 
so that
\begin{equation}  \label{Low4} 
\inf  _{ x\in \R^N } \langle x\rangle ^\Imqd  | v(t,x) | \ge \inf  _{ x\in \R^N  } \langle x\rangle ^\Imqd   |\psi  (x)| - t  \| \langle x\rangle ^\Imqd  \Delta v \| _{ L^\infty ((0, t  ) \times \R^N ) } .
\end{equation} 
We now must estimate the last  term on the right-hand side of~\eqref{Low4}. Note that we cannot simply use Sobolev's embedding $H^s\hookrightarrow L^\infty $ for $s>\frac {N} {2}$.
Indeed, this would require $\langle x\rangle ^\Imqd  \Delta \psi \in L^2 (\R^N ) $, i.e. $ \langle x\rangle ^{-2 }\in L^2 (\R^N ) $, which fails if $N\ge 4$. On the other hand, we note that  $ |\langle x\rangle ^\Imqd  \Delta^{k+1} \psi | \le C \langle x\rangle ^{-2 k -2}$, which belongs to $L^2 (\R^N ) $ if $k$ is sufficiently large. Therefore, instead of applying~\eqref{fInt1}, we apply Taylor's formula with integral remainder involving derivatives of $v$ of sufficiently large order, and this leads to estimating $\langle x\rangle ^\Imqd  \Delta^{k+1} v(t)$ in the Sobolev space $H^s $ where $s>\frac {N} {2}$ and $k$ is sufficiently large. This first step is achieved in Lemma~\ref{eLE1} below. In order to estimate $ \| \langle x\rangle ^\Imqd  \Delta^{k+1} v(t) \| _{ H^s }$, we use  energy estimates. Every integration by parts will decrease by $1$ the power of $ \langle x\rangle $ which is involved in the estimate, but will at the same time increase by $1$ the number of derivatives. 
This second step is achieved in Lemma~\ref{eLE3:0} below, and this explains why the definition of the space $\Spa$ involves weighted $L^\infty $-norms of the derivatives of the function up to a certain order, then weighted $L^2 $-norms of the derivatives of higher order. 
The combination of Lemmas~\ref{eLE1} and~\ref{eLE3:0} yields Proposition~\ref{eLE3} below, which is the main linear estimate we use in this paper. It shows in particular that for $\psi $ as above,  $\inf  _{ x\in  \R^N }  \langle x\rangle ^\Imqd  | v (t, x) | $ remains positive for all sufficiently small $t$. 
The proof of Theorem~\ref{eThm1} is then a simple contraction mapping argument applied to equation~\eqref{NLS1}. This argument requires, as usual, a linear estimate (Proposition~\ref{eLE3} is our case) and a nonlinear estimate. The nonlinear estimate is provided by Proposition~\ref{eNL1} below, which yields an estimate of $ |u|^\alpha u$ in the space $\Spa$, assuming $u\in \Spa$ satisfies $ |u (x)| \ge c \langle x\rangle ^{- \Imqd }$ for some $c>0$. This justifies the introduction of the space $\Spa$, which is well-suited for both the Schr\"o\-din\-ger group and the nonlinearity. 

The proof of Theorem~\ref{eThm2} requires one more argument, which is inspired from~\cite{CWrapdec}. It consists in applying the pseudo-conformal transformation to equation~\eqref{NLS1}. A global solution of~\eqref{NLS1} which scatters, corresponds to a solution of the non-autonomous equation~\eqref{NLS2} which is defined on $[0, \frac {1} {b}]$. 
If $\alpha >\frac {2} {N}$, then the time-dependent factor in~\eqref{NLS2} is integrable at $\frac {1} {b}$, so that can apply a standard contraction argument (based on Propositions~\ref{eLE3} and~~\ref{eNL1}) to construct a solution on $[0, \frac {1} {b})$, provided $b$ is sufficiently large.

The rest of this paper is organized as follows. In Section~\ref{sLow} we establish Proposition~\ref{eLE3}, which measures the action of the Schr\"o\-din\-ger  group $(e^{it\Delta }) _{ t\in \R }$ on the space $\Spa$. Section~\ref{sNL} is devoted to the nonlinear estimate, i.e. the estimate of $ |u|^\alpha u$ in $\Spa$. The proofs of Theorems~\ref{eThm1} and~\ref{eThm2} are completed in Section~\ref{sNLS}. Finally, we recall in Appendix~\ref{sElem} some elementary estimates which we use in the paper. 

\begin{notation} 
We denote by $L^p (U) $, for $1\le p\le \infty $ and $U=\R^N $ or $U= (0,T)\times \R^N $, $0<t\le \infty $, the usual (complex valued) Lebesgue spaces. 
We use the standard notation that $ \| u \| _{ L^p }=\infty $ if $u\in L^1_\Loc (U )$ and $u\not \in L^p (U) $.
$H^s (\R^N  ) $,  $s\in \R$, is the usual (complex valued) Sobolev space.  (See e.g.~\cite{AdamsF} for the definitions and properties of these spaces.)
We denote by $(e^{it \Delta }) _{ t\in \R }$ the Schr\"o\-din\-ger group on $\R^N $. As is well known, $(e^{it \Delta }) _{ t\in \R }$ is a group of isometries on $L^2 (\R^N ) $, and on $H^s (\R^N ) $ for all $s\in \R$. 

\end{notation} 

\section{Weighted estimates for the linear Schr\"o\-din\-ger equation} \label{sLow} 

Our main result of this section is the following estimate of the action of the Schr\"o\-din\-ger  group $(e^{it\Delta }) _{ t\in \R }$ on the space $\Spa$. 

\begin{prop} \label{eLE3}
Assume~\eqref{fDInt1}-\eqref{fSpa1b1} with $\alpha =1$, and let the space
$ \Spa$ be defined by~\eqref{fSpa1}-\eqref{fSpa2}.
Given $\psi \in \Spa$, it follows that $e^{it \Delta } \psi \in C ( \R , \Spa )$. 
Moreover, there exists a constant $C$ such that
\begin{equation} \label{eLE3:11}
 \| e^{it \Delta }  \psi \|_\Spa  \le C (1+ |t|)^{\Imqu + \Imqd +1}   \| \psi  \|_\Spa 
\end{equation} 
for all $t\in \R$ and all $\psi \in \Spa$. In addition,
\begin{equation} \label{eLE3:12}
\sup   _{ |\beta |\le 2 \Imqu }
 \| \langle x\rangle ^{\Imqd} D^\beta  (e^{it \Delta }  \psi -\psi ) \| _{ L^\infty  } \le C  |t| (1+ |t|)^{\Imqu + \Imqd +1}   \| \psi  \|_\Spa 
\end{equation} 
for all $t \in \R$ and all $\psi \in \Spa$. 
 \end{prop} 
 
 Before proving Proposition~\ref{eLE3}, we first establish the following weighted $L^\infty $ estimate.

\begin{lem} \label{eLE1} 
Assume~\eqref{fDInt1}-\eqref{fSpa1b1} with $\alpha =1$. There exists a constant $C$ such that
\begin{equation} \label{Low5b1} 
\begin{split} 
 \sum_{ \Imqc =0 }^{2\Imqu  }   \sup _{  |\beta |=    \Imqc }  \| \langle x\rangle ^\Imqd  D^\beta  e^{i s \Delta } &\psi  \| _{ L^\infty ((0,t ) \times  \R^N ) } \le 
C (1+t)^{\Imqu } \sum_{ \Imqc =0 }^{2\Imqu  }   \sup _{  |\beta |=    \Imqc }  \| \langle x\rangle ^\Imqd D^\beta  \psi  \| _{ L^\infty  } \\ &  + C t (1+t)^{\Imqu }\sup _{  |\beta |=   2 \Imqu  +2}  \| \langle x\rangle ^\Imqd  D^\beta e^{i s \Delta }\psi  \| _{ L^\infty ((0, t ) \times \R^N ) }.
\end{split} 
\end{equation} 
for all $t \ge 0$ and all $\psi \in H^\Imqh (\R^N )$.
\end{lem} 

\begin{proof} 
Set  $v (t) = e^{i t\Delta }\psi $. Since $\Delta ^\Imqc \psi \in H^{\Imqh- 2 \Imqc} (\R^N ) $ for $0\le \Imqc \le \Imqu+1 $, we have $v\in C^{\Imqc} ([0,\infty ), H^{\Imqh- 2 \Imqc}  (\R^N ) )$ and $\frac {d^\Imqc v} {dt^\Imqc}= i^{\Imqc} \Delta ^\Imqc v(t)$ for all $0\le \Imqc\le \Imqu+1 $. 
Given $0\le \Imtq \le \Imqu$, we apply 
 Taylor's formula with integral remainder involving the derivative of order $\Imqu - \Imtq +1$ to the function $v$, and we obtain
\begin{equation} \label{eAbs1:1} 
v(t)  = \sum_{ \Imqc=0 }^{\Imqu - \Imtq}  \frac {(it)^\Imqc} {\Imqc !} \Delta ^\Imqc \psi  + \frac {i ^{\Imqu - \Imtq +1}} {(\Imqu - \Imtq) ! } \int _0^t (t-s)^{\Imqu - \Imtq} \Delta ^{\Imqu - \Imtq +1} v(s) \, ds
\end{equation} 
for all $t\ge 0$. 
Applying now $D^\beta $ with $ |\beta | = 2  \Imtq $, we deduce that 
\begin{equation} \label{eAbs1:2} 
D^\beta v(t)  = \sum_{ \Imqc=0 }^{\Imqu - \Imtq}  \frac {(it)^\Imqc} {\Imqc !} D^\beta  \Delta ^\Imqc \psi  + \frac {i ^{\Imqu - \Imtq +1}} {(\Imqu - \Imtq) ! } \int _0^t (t-s)^{\Imqu - \Imtq} D^\beta  \Delta ^{\Imqu - \Imtq +1} v(s) \, ds .
\end{equation} 
Identity~\eqref{eAbs1:2} holds in $C([0,\infty ), H^{\Imqt} (\R^N ) )$, hence in $C([0,\infty )\times \R^N )$ by Sobolev's embedding $H^\Imqt \subset  C(\R^N )$. Multiplying by $\langle x\rangle ^\Imqd  $ and taking the supremum in $x$, then in $t$, we obtain
\begin{equation} \label{eAbs1:3} 
\begin{split} 
 \sup _{  |\beta |=     2  \Imtq  }  \| \langle x\rangle ^\Imqd  D^\beta  e^{i s \Delta } &\psi  \| _{ L^\infty ((0,t ) \times  \R^N ) } \le 
C (1+t)^{\Imqu } \sum_{ \Imqc =0 }^{2\Imqu  }   \sup _{  |\beta |=    \Imqc }  \| \langle x\rangle ^\Imqd D^\beta  \psi  \| _{ L^\infty  } \\ &  + C t (1+t)^{\Imqu }\sup _{  |\beta |=   2 \Imqu  +2}  \| \langle x\rangle ^\Imqd  D^\beta e^{i s \Delta }\psi  \| _{ L^\infty ((0, t ) \times \R^N ) }.
\end{split} 
\end{equation} 
Since the right-hand side of the above inequality is independent of $0\le \Imtq \le \Imqu$, it follows by summing up in $\Imtq$ that
\begin{equation} \label{eAbs1:4} 
\begin{split} 
\sum_{ \Imqc =0 }^{\Imqu  }  \sup _{  |\beta |=     2  \Imqc  }  \| \langle x\rangle ^\Imqd  D^\beta  e^{i s \Delta } &\psi  \| _{ L^\infty ((0,t ) \times  \R^N ) } \le 
C (1+t)^{\Imqu } \sum_{ \Imqc =0 }^{2\Imqu  }   \sup _{  |\beta |=    \Imqc }  \| \langle x\rangle ^\Imqd D^\beta  \psi  \| _{ L^\infty  } \\ &  + C t (1+t)^{\Imqu }\sup _{  |\beta |=   2 \Imqu  +2}  \| \langle x\rangle ^\Imqd  D^\beta e^{i s \Delta }\psi  \| _{ L^\infty ((0, t ) \times \R^N ) }.
\end{split} 
\end{equation} 
By the interpolation estimate~\eqref{fSob6}, derivatives of odd order in the left-hand side of~\eqref{Low5b1} are estimated by the left-hand side of~\eqref{eAbs1:4}, and we conclude that~\eqref{Low5b1} holds.
\end{proof} 

We now want to estimate the last term in the right-hand side of~\eqref{Low5b1} by Sobolev's embebbing. In order to do this, we establish a weighted $L^2$ estimate (Lemma~\ref{eLE3:0} below), for which we introduce the following notation. Assuming~\eqref{fDInt1}-\eqref{fSpa1b1} with $\alpha =1$, we define the space
\begin{equation} \label{fSpa1:0} 
\begin{split} 
\Spb=  \{ u\in H^\Imqh  (\R^N ); & \,  \langle x\rangle ^\Imqd D^\beta u \in L^2  (\R^N )  \text{ for  }   2\Imqu +1 \le  |\beta | \le 2\Imqu +2 + \Imqt,  \\ & \langle x\rangle ^{\Imqh -  |\beta |} D^\beta u \in L^2  (\R^N )  \text{ for  }    2\Imqu +2 + \Imqt <  |\beta | \le J \} 
\end{split} 
\end{equation} 
and we equip $\Spb$ with the  norm
\begin{equation} \label{fSpa2:0}
 \| u \|_\Spb =  \sum_{ \Imqc =0 }^{2\Imqu  }   \sup _{  |\beta |=    \Imqc }  \|  D^\beta  u  \| _{ L^2  } +   \sum_{\Imqp =0 }^{\Imqt +1} \sum_{ \Imqs =0 }^\Imqd \sum_{  |\beta  |=\Imqp + \Imqs  +  2 \Imqu  +1  }   \| \langle x \rangle ^{\Imqd -\Imqs } D^\beta  u \| _{ L^2 } .
\end{equation}
Note that by~\eqref{fSpa1:b2},  $\Spa \hookrightarrow \Spb$. 
Standard considerations show that $(\Spb,  \| \cdot  \|_\Spb )$ is a Banach space and that $\Srn$ is dense in $\Spb$.

\begin{lem} \label{eLE3:0}
Assume~\eqref{fDInt1}-\eqref{fSpa1b1}  with $\alpha =1$, and let the space
$ \Spb $ be defined by~\eqref{fSpa1:0}-\eqref{fSpa2:0}.
Given $\psi \in \Spb$, it follows that $e^{it \Delta } \psi \in C ( [0,\infty ), \Spb )$. 
Moreover, there exists a constant $C$ such that
\begin{equation} \label{eLE3:11:0}
 \| e^{it \Delta }  \psi \|_\Spb  \le C (1+t)^{  \Imqd }   \| \psi  \|_\Spb 
\end{equation} 
for all $t\ge 0$ and all $\psi \in \Spb$.
 \end{lem} 
 
\begin{proof} 
Given $\psi \in \Srn$, we have $e^{it \Delta }  \psi \in C([0,\infty ), \Srn ) \subset C([0, \infty ), \Spb )$.
Therefore, by density of $\Srn$ in $\Spb$, the result follows if we prove estimate~\eqref{eLE3:11:0} for $\psi \in \Srn$. 
So we consider $\psi \in \Srn$ and we set $v (t) = e^{it\Delta } \psi $. 
Since the Schr\"o\-din\-ger group is isometric on $H^{2\Imqu} (\R^N ) $, we need only estimate the weighted terms in $ \| e^{it \Delta }  \psi \|_\Spb$.  
We claim that, given any $\ell \in \N$,
\begin{equation} \label{fSobz2} 
 \| \langle x \rangle ^{\Intq } v(t) \| _{ L^2 } \le C  (1+t )^\Intq  \sum_{ \Ints =0 }^\Intq  \sum_{  |\beta  |=\Ints   }   \| \langle x \rangle ^{\Intq -\Ints } D^\beta  \psi \| _{ L^2 } 
 \end{equation} 
 where $C$ is independent of $\psi $. 
We note that~\eqref{fSobz2} is immediate for $\Intq =0$ (since $(e^{it \Delta }) _{ t\in \R }$ is isometric on $L^2 (\R^N ) $). We now proceed by induction on $\Intq $, so we suppose~\eqref{fSobz2} holds up to some $\Intq \ge 1$. 
We prove that
\begin{equation} \label{fSobz1} 
 \| \langle x \rangle ^{\Intq +1} v(t) \| _{ L^2  } \le  \| \langle x \rangle ^{\Intq +1} \psi \| _{ L^2 } +2(\Intq +1) t  \| \langle x \rangle ^{\Intq } \nabla v \| _{ L^\infty ( (0, t ), L^2 ) } .
\end{equation} 
This follows from an elementary integration by parts. 
Indeed, multiplying~\eqref{Low2} by $ \langle x \rangle ^{2\Intq +2} \overline{v} $ and taking the imaginary part, we obtain after integration by parts over $\R^N $
\begin{equation*} 
\begin{split} 
\frac {1} {2}\frac {d} {dt} \| \langle x \rangle ^{\Intq +1} v(t) \|  _{ L^2 }^2 & = \Im \int  _{ \R^N  } \nabla v \cdot \nabla (   \langle x \rangle ^{2\Intq +2}  \overline{v} ) = \Im \int  _{ \R^N  }  \overline{v}  \nabla v \cdot \nabla ( \langle x \rangle ^{2\Intq +2} ) \\ & \le 2(\Intq +2)   \int  _{ \R^N  }  \langle x \rangle ^{2\Intq +1}  |v|\,  |\nabla v|
\\ & \le  2(\Intq +2)   \| \langle x \rangle ^{\Intq  +1} v(t) \|  _{ L^2 }  \| \langle x \rangle ^{\Intq } \nabla v(t) \|  _{ L^2 }
\end{split} 
\end{equation*} 
where we used~\eqref{fSob4} in the next-to-last inequality,  and~\eqref{fSobz1} follows.
We now estimate the last term on the right-hand side of~\eqref{fSobz1} by applying~\eqref{fSobz2} at the level $\Intq $, but with $\psi $ replaced by $\nabla \psi $, and we obtain
\begin{equation*} 
 \| \langle x \rangle ^{\Intq +1} v(t)  \| _{ L^2 } \le  \| \langle x \rangle ^{\Intq +1} \psi \| _{ L^2 } +   C t (1+t )^\Intq  \sum_{ \Ints =0 }^\Intq  \sum_{  |\beta  |=\Ints  +1 }   \| \langle x \rangle ^{\Intq -\Ints } D^\beta  \psi \| _{ L^2 }  .
\end{equation*} 
Thus~\eqref{fSobz2} holds at the level $\Intq +1$, which proves~\eqref{fSobz2} for all $\Intq\ge 0$. 

We now fix $0\le \Imqp \le \Imqt +1$, $0\le \Imqs\le \Imqd$, we consider a multi-index $\beta $ such that $ |\beta | = \Imqp + \Imqs + 2 \Imqu  +1$, and we apply~\eqref{fSobz2} with $\psi $ replaced by $D^\beta \psi $, and $ \Intq= \Imqd- \Imqs$. It follows that 
\begin{equation} \label{fSobz2:1} 
\begin{split} 
 \| \langle x \rangle ^{\Imqd- \Imqs } D^\beta  v(t) \| _{ L^2 } & \le C  (1+t )^{\Imqd }  \sum_{ \Intc =0 }^{\Imqd- \Imqs}  \sum_{  |\gamma   |=\Intc  + \Imqp + \Imqs + 2 \Imqu  +1  }   \| \langle x \rangle ^{\Imqd- \Imqs -\Intc } D^\gamma   \psi \| _{ L^2 }  \\ &  = C  (1+t )^{\Imqd }  \sum_{ \Intc =\Imqs }^{\Imqd}  \sum_{  |\gamma   |=\Intc  + \Imqp  + 2 \Imqu  +1  }   \| \langle x \rangle ^{\Imqd -\Intc } D^\gamma   \psi \| _{ L^2 }   \\ & \le  C  (1+t )^{\Imqd }  \sum_{ \Intc =0 }^{\Imqd}  \sum_{  |\gamma   |=\Intc  + \Imqp  + 2 \Imqu  +1  }   \| \langle x \rangle ^{\Imqd -\Intc } D^\gamma   \psi \| _{ L^2 } \\ & \le C  (1+t )^{\Imqd }  \| \psi  \|_\Spb .
\end{split} 
 \end{equation} 
Thus we see that every weighted term in $ \| v(t)  \|_\Spb$ is estimated by $C  (1+t )^{\Imqd }  \| \psi  \|_\Spb $. This shows~\eqref{eLE3:11:0} and completes the proof. 
\end{proof} 

\begin{proof}[Proof of Proposition~$\ref{eLE3}$] 
Since $e^{-it \Delta }\psi =  \overline{e^{it\Delta }\psi } $ and the map $\psi \mapsto  \overline{\psi } $ is  isometric $\Spa \to \Spa$, we need only establish the various properties for $t\ge 0$.

Let $\psi \in \Spa$ and set $v(t)= e^{it\Delta } \psi $. We first prove that $v(t) \in \Spa$ for all $t\ge 0$ and~\eqref{eLE3:11} holds. It follows from~\eqref{Low5b1} and~\eqref{fSpa2}  that
\begin{equation} \label{Low5b2} 
\begin{split} 
 \sum_{ \Imqc =0 }^{2\Imqu  }   \sup _{  |\beta |=    \Imqc }  \| & \langle x\rangle ^\Imqd  D^\beta  v   \| _{ L^\infty ((0,t ) \times  \R^N ) } \le 
C (1+t)^{\Imqu }  \| \psi \|_\Spa \\ &  + C t (1+t)^{\Imqu }\sup _{  |\beta |=   2 \Imqu  +2}  \| \langle x\rangle ^\Imqd  D^\beta v \| _{ L^\infty ((0, t ) \times \R^N ) }.
\end{split} 
\end{equation} 
We apply Sobolev's embedding $H^\Imqt   \hookrightarrow L^\infty   $ to the last term in the right-hand side of~\eqref{Low5b2}. Given $2 \Imqu  +1 \le \Imqs \le  2 \Imqu  +2$ we have
\begin{equation} \label{eLE3:1}
\sup _{  |\beta |= \Imqs}  \| \langle x\rangle ^\Imqd  D^\beta v  \| _{ L^\infty ((0, t ) \times \R^N ) } \le C \sup  _{ \substack {0\le s\le t \\   |\beta | = \Imqs} }  \| \langle x\rangle ^\Imqd  D^\beta v(s)  \| _{ H^\Imqt } .
\end{equation} 
Note that
\begin{equation*}
 \| \langle x\rangle ^\Imqd  D^\beta v(s) \| _{ H^\Imqt } \le C \sum_{  |\gamma |\le \Imqt } \| D^\gamma ( \langle x\rangle ^\Imqd  D^\beta v(s)  ) \| _{ L^2 } 
\end{equation*} 
and that, by~\eqref{fSob5},
\begin{equation*}
 | D^\gamma  (\langle x\rangle ^\Imqd D^\beta v(s)  ) ) |  \le C \sum_{ \Imqq =0 }^{ |\gamma  |}   \sum_{  |\rho   |=\Imqq } \langle x\rangle ^{ \Imqd -  |\gamma  |+\Imqq } |D^{ \rho   + \beta} v(s)  | 
 \le C   \sum_{  |\rho   |\le  |\gamma | } \langle x\rangle ^{ \Imqd  } |D^{ \rho   + \beta} v(s)  | .
\end{equation*} 
Therefore 
\begin{equation} \label{eLE3:4}
\begin{split} 
\sup _{  |\beta | = \Imqs} \| \langle x\rangle ^\Imqd  D^\beta v(s)  \| _{ H^\Imqt } &  \le C \sup _{  |\beta | = \Imqs}  \sum_{  |\rho   |\le \Imqt }  \|   \langle x\rangle ^{ \Imqd  } D^{ \rho   + \beta} v(s) \| _{ L^2 } \\   \le C   \sum_{  |\beta    | = \Imqs } ^{ \Imqs +\Imqt} \|   \langle x & \rangle ^{ \Imqd } D^{ \beta  } v(s) \| _{ L^2 }  
 \le C  \| v(s)  \|_\Spb  \le C (1+s)^{  \Imqd }   \| \psi  \|_\Spb  
\end{split} 
\end{equation} 
where we used~\eqref{eLE3:11:0} in the last inequality. 
Estimate~\eqref{eLE3:11} (along with the property $v(t) \in \Spa$ for $t\ge 0$) follows from~\eqref{eLE3:11:0}, \eqref{Low5b2}, \eqref{eLE3:1}  and~\eqref{eLE3:4} (applied with $\Imqs =  2 \Imqu  +2$). 

Next, consider a multi-index $\beta $ with $ |\beta |\le 2 \Imqu$. 
If $ |\beta |+2 \ge 2 \Imqu +1$, then it follows from~\eqref{fInt1}, \eqref{eLE3:1}  and~\eqref{eLE3:4} (applied with $\Imqs =   |\beta |  +2$) that
\begin{equation*} 
 \| \langle x\rangle ^\Imqd D^\beta  ( v(t) - \psi ) \| _{ L^\infty  } \le \int _0^t  \| \langle x\rangle ^\Imqd D^\beta   \Delta v(s) \|  _{ L^\infty  } ds \le C t (1+t)^{  \Imqd }   \| \psi  \|_\Spb  .
\end{equation*} 
If $ |\beta |+2 \le 2 \Imqu $, then $ \| \langle x\rangle ^\Imqd D^\beta   \Delta v(s) \|  _{ L^\infty  } \le  \| v(s) \|_\Spa $. Therefore, it follows from~\eqref{fInt1} and~\eqref{eLE3:11} that $  \| \langle x\rangle ^\Imqd D^\beta  ( v(t) - \psi ) \| _{ L^\infty  } \le C t (1+t)^{\Imqu + \Imqd +1} $, which completes the proof of~\eqref{eLE3:12}.

We finally show that $v\in C ( [0,\infty ), \Spa )$.
By the semigroup property, we need only show continuity at $t=0$. 
Moreover, $v \in C ( [0,\infty ), \Spb )$ by Lemma~\ref{eLE3:0}, so we need only estimate the terms involving $L^\infty $ norms. This follows from~\eqref{eLE3:12}. 
\end{proof}  

\begin{rem} \label{eRem4} 
It follows in particular from Proposition~\ref{eLE3} that if $\psi \in \Spa$ satisfies $\inf  _{ x\in \R^N  } \langle x\rangle ^\Imqd  |\psi (x)| >0$, then $\inf  _{ x\in \R^N  } \langle x\rangle ^\Imqd  |e^{it \Delta }\psi (x)| >0$ provided $ |t|$ is sufficiently small. We do not know if this small time requirement is necessary.
\end{rem} 

\section{A nonlinear estimate} \label{sNL} 

We establish an estimate of $  |u|^\alpha u$ in the space $\Spa$.

\begin{prop} \label{eNL1} 
Let $\alpha >0$. 
Assume~\eqref{fDInt1}-\eqref{fSpa1b1}, and let the space $ \Spa$ be defined by~\eqref{fSpa1}-\eqref{fSpa2}. For every $\eta >0$  and $u\in \Spa $ such that 
\begin{equation} \label{eNL1:2} 
\eta \inf _{ x\in \R^N  } ( \langle x\rangle ^\Imqd  |u(x)|)  \ge 1
\end{equation} 
it follows that $ |u|^\alpha u\in \Spa$. Moreover, there exists a constant $C$ such that 
\begin{equation} \label{eNL1:1} 
 \| \, |u|^\alpha u \|_\Spa  \le C  (1+ \eta  \| u \|_\Spa)^{2 \Imqh}   \| u \|_\Spa ^{\alpha +1}
\end{equation} 
for all $\eta >0$  and $u\in \Spa $ satisfying~\eqref{eNL1:2}. In addition,
\begin{equation} \label{eNL1:2b2} 
\begin{split} 
 \| \, |u_1 |^\alpha u_1 &- |u_2 |^\alpha u_2 \|_\Spa \\ & \le C  \bigl((1+ \eta (  \| u_1 \|_\Spa  +  \| u_2 \|_\Spa ) \bigr)^{2 \Imqh +1} ( \| u_1\|_\Spa +  \|u_2\| _{ \Spa  })^\alpha  \| u_1 - u_2 \|_\Spa
\end{split} 
\end{equation} 
for all $\eta >0$  and $u_1, u_2\in \Spa $ satisfying~\eqref{eNL1:2}. 
\end{prop} 

\begin{proof} 
We first calculate $D^\beta (|u|^\alpha u)$ with $1\le   |\beta |\le \Imqh$. 
We observe that 
\begin{equation} \label{eNL1:2b1} 
D^\beta (|u|^\alpha u) =   \sum_{  \gamma +\rho = \beta  } 
c _{ \gamma ,\rho  } D^\gamma (|u|^\alpha ) D^\rho u,
\end{equation} 
with the coefficients $c _{ \gamma ,\rho  }$ given by Leibniz's rule. 
Since $ |u|^\alpha = (u  \overline{u})^{\frac {\alpha } {2}} $ we see that the development of $D^\beta (|u|^\alpha u) $ contains on the one hand the term 
\begin{equation} \label{eNL1:3} 
A=   |u|^\alpha D^\beta u 
\end{equation} 
and on the other hand, terms of the form
\begin{equation} \label{eNL1:4} 
B=  |u|^{\alpha -2p} D^\rho u \prod  _{ j=1 }^p D^{\gamma _{1, j}} u D^{\gamma _{2, j}} \overline{u} 
\end{equation} 
where 
\begin{equation} \label{eNL1:5} 
\gamma +\rho =\beta ,  \quad 1\le p \le  |\gamma |, \quad  |\gamma  _{ 1,j } + \gamma  _{ 2,j }| \ge 1,  \quad \sum_{ j=0 }^p (\gamma  _{ 1,j } + \gamma  _{ 2,j }) = \gamma .
\end{equation} 
We now proceed in two steps.

\bigskip 

\noindent  {\bf Step 1.}\quad Proof of~\eqref{eNL1:1}.  
If $ |\beta |\le 2 \Imqu$, we need to estimate the terms $\langle x\rangle ^\Imqd A$ and $\langle x\rangle ^\Imqd B$ in $L^\infty $. 
If $2 \Imqu + 1 \le  |\beta | \le 2 \Imqu + 2 + \Imqt$, we need to estimate the terms $\langle x\rangle ^\Imqd A$ and $\langle x\rangle ^\Imqd B$ in $L^2 $. 
If $ 2 \Imqu + 3 + \Imqt \le  |\beta | \le  \Imqh$, we need to estimate the terms $\langle x\rangle ^{\Imqh -  |\beta |} A$ and $\langle x\rangle ^{\Imqh -  |\beta |} B$ in $L^2 $.

We note that the term~\eqref{eNL1:3} is very easy to handle, and gives contributions estimated  by $  \|u\| _\Spa^\alpha   \| u \|_\Spa $, hence by the right-hand side of~\eqref{eNL1:1}. 

We now  concentrate on the terms~\eqref{eNL1:4} and we observe that, due to the lower bound~\eqref{eNL1:2}
\begin{equation} \label{eNL1:6b1} 
 |u|^{\alpha -2p} \le \eta ^{2p} \langle x\rangle ^{2 p \Imqd }   |u|^{\alpha } \le  \eta ^{2p}  \langle x\rangle ^{(2 p - \alpha ) \Imqd } \| u \|_\Spa ^{\alpha }   .
\end{equation} 
so that
\begin{equation} \label{eNL1:6}  
 |B |   \le  \eta ^{2p} \langle x\rangle ^{2 p \Imqd }   |u|^{\alpha } \le  \eta ^{2p}  \langle x\rangle ^{(2 p - \alpha ) \Imqd } \| u \|_\Spa ^{\alpha }    |D^\rho u| \prod  _{ j=1 }^p  |D^{\gamma _{1, j}} u |\,  |D^{\gamma _{2, j}} u| .
\end{equation} 
We now consider three different cases.

\medskip 

\noindent {\bf Case 1.}\quad Suppose $ |\beta |\le 2 \Imqu$, so that we need to estimate $ \|  \langle x\rangle ^{ \Imqd }  B \| _{ L^\infty  }  $. It follows that all the derivatives in the right-hand side of~\eqref{eNL1:6} are also of order $\le 2 \Imqu$, hence estimated by $\langle x\rangle ^{- \Imqd }   \| u \|_\Spa $; and so~\eqref{eNL1:6} yields 
\begin{equation}  \label{eNL1:7} 
 |B | \le  
 (   \eta \| u \|_\Spa )^{2p}  \langle x\rangle ^{- (\alpha +1)\Imqd } 
   \| u \|_\Spa ^{\alpha  +1}.
\end{equation} 
Therefore,
\begin{equation} \label{eNL1:8} 
 \|  \langle x\rangle ^{ \Imqd }  B \| _{ L^\infty  }   \le  (   \eta \| u \|_\Spa )^{2p} 
   \| u \|_\Spa ^{\alpha  +1}
   \end{equation} 
which is controlled by the right-hand side of~\eqref{eNL1:1}. 

\medskip 

\noindent {\bf Case 2.}\quad Suppose $2 \Imqu + 1 \le  |\beta | \le 2 \Imqu + 2 + \Imqt$, so that we need to estimate $ \|  \langle x\rangle ^{ \Imqd }  B \| _{ L^2 }  $.
Assume one of the derivatives in the right-hand side of~\eqref{eNL1:6}  is of order $\ge 2 \Imqu + 1$, for instance  $ |\gamma  _{ 1,1 } |\ge 2 \Imqu + 1$. 
Since the sum of all derivatives has order $ |\beta |$, and $4 \Imqu + 2 >  2 \Imqu + 2 + \Imqt \ge  |\beta |$ (by the third inequality in~\eqref{fDInt1}), it follows that all other derivatives have order $\le 2 \Imqu$, hence are estimated by  $\langle x\rangle ^{- \Imqd }   \| u \|_\Spa $. Therefore, \eqref{eNL1:6} yields
\begin{equation}  \label{eNL1:9} 
 |B | \le   (   \eta \| u \|_\Spa )^{2p}  \langle x\rangle ^{ - \alpha  \Imqd } \| u \|_\Spa ^{\alpha }      |D^{\gamma _{1, 1}} u | .
\end{equation}
Since $ 2 \Imqu + 1 \le  |\gamma  _{ 1,1 } | \le  |\beta |\le 2 \Imqu + 2 + \Imqt $, we have $ \| \langle x\rangle ^{ \Imqd } D^{\gamma  _{ 1,1 }} u \| _{ L^2 } \le \| u \|_\Spa $, so we see that $ \| \langle x\rangle ^{ \Imqd } B \| _{ L^2 } $ is 
estimated by the right-hand side of~\eqref{eNL1:1}. 
If all the derivatives in the right-hand side of~\eqref{eNL1:6}  are of order $\le 2 \Imqu $, then they are estimated by  $\langle x\rangle ^{- \Imqd }  \| u \|_\Spa $, and we obtain again~\eqref{eNL1:7}, which yields 
\begin{equation}  \label{eNL1:10} 
 \langle x\rangle ^{ \Imqd }   |B |  \le (   \eta \| u \|_\Spa )^{2p}  \langle x\rangle ^{- \alpha \Imqd } 
   \| u \|_\Spa ^{\alpha  +1} .
\end{equation} 
Since $ \alpha  \Imqd > \frac {N} {2}$ by the second inequality in~\eqref{fDInt1}, the right-hand side of the above inequality belongs to $L^2 (\R^N ) $, and we obtain again an estimate by the right-hand side of~\eqref{eNL1:1}. 

\medskip 

\noindent {\bf Case 3.}\quad Suppose $ 2 \Imqu + 3 + \Imqt \le  |\beta | \le  \Imqh$, so that we need to estimate $\langle x\rangle ^{\Imqh -  |\beta |} B$ in $L^2 $. This is very similar to Case~2. Assume one of the derivatives in the right-hand side of~\eqref{eNL1:6}  is of order $\ge 2 \Imqu + 1$, for instance  $ |\gamma  _{ 1,1 } |\ge 2 \Imqu + 1$. 
Since the sum of all derivatives has order $ |\beta |$, and $4 \Imqu + 2 > \Imqh \ge  |\beta |$ (by the third inequality in~\eqref{fDInt1}), it follows that all other derivatives have order $\le 2 \Imqu$, hence are estimated by  $\langle x\rangle ^{- \Imqd } \| u \|_\Spa $. Therefore, \eqref{eNL1:6} yields
estimate~\eqref{eNL1:9}.  
If $ 2 \Imqu + 1 \le  |\gamma  _{ 1,1 } | \le   2 \Imqu + 2 + \Imqt $, we have $ \| \langle x\rangle ^{\Imqh -  |\beta |} D^{\gamma  _{ 1,1 }} u \| _{ L^2 }\le  \| \langle x\rangle ^{ \Imqd } D^{\gamma  _{ 1,1 }} u \| _{ L^2 } \le  \| u \|_\Spa $, so we deduce from~\eqref{eNL1:9} that $ \| \langle x\rangle ^{\Imqh -  |\beta |} B \| _{ L^2 } $ is 
estimated by the right-hand side of~\eqref{eNL1:1}. 
If $2 \Imqu + 3 + \Imqt \le |\gamma  _{ 1,1 } | \le  |\beta |$, then $  \| \langle x\rangle ^{\Imqh -  |\beta |} D^{\gamma  _{ 1,1 }} u \| _{ L^2 }\le  \| \langle x\rangle ^{ \Imqh - |\gamma  _{ 1,1 }|} D^{\gamma  _{ 1,1 }} u \| _{ L^2 } \le  \| u \|_\Spa $, and we conclude as just above. 
Finally, if all the derivatives in the right-hand side of~\eqref{eNL1:6}  are of order $\le 2 \Imqu $, then they are estimated by  $\langle x\rangle ^{- \Imqd }  \| u \|_\Spa $, and we obtain again~\eqref{eNL1:7}, which yields estimate~\eqref{eNL1:10}. Since $ \alpha  \Imqd > \frac {N} {2}$ by the second inequality in~\eqref{fDInt1}, the right-hand side of~\eqref{eNL1:10} belongs to $L^2 (\R^N ) $, and we obtain again an estimate by the right-hand side of~\eqref{eNL1:1}. 
This completes the proof of~\eqref{eNL1:1}.  

\medskip 

\noindent  {\bf Step 2.}\quad Proof of~\eqref{eNL1:2b2}.  We use the expressions~\eqref{eNL1:3} and~\eqref{eNL1:4} for both $u_1$ and $u_2$ and we form the difference. 
Suppose for instance that 
\begin{equation} \label{fSPb1} 
 \| u_2 \|_\Spa  \le  \| u_1 \|_\Spa .
\end{equation} 
Concerning~\eqref{eNL1:3}, this yields $ |u_1 |^\alpha D^\beta u_1 -  |u_2|^\alpha D^\beta u_2 $, which we write  $ |u_1|^\alpha (D^\beta u_1 - D^\beta u_2 ) + ( |u_1|^\alpha-  |u_2|^\alpha ) D^\beta u_2 $.  
Arguing as in Step~1, we see that the first term is estimated in the appropriate weighted spaces by $  \|u_1 \| _{ L^\infty  }^\alpha   \| u_1 -u_2 \|_\Spa $, hence by the right-hand side of~\eqref{eNL1:2b2}. The second term is estimated by $  \| \,  |u_1|^\alpha -  |u_2|^\alpha \| _{ L^\infty  } \| u_2 \|_\Spa $. 
We note that by~\eqref{eNL1:2} and~\eqref{fSPb1}
\begin{equation*} 
\begin{split} 
| \,  |u_1 |^\alpha -  |u_2 |^\alpha | & \le C( |u_1|^{-1 } + |u_2|^{-1 })  ( |u_1| +  |u_2|)^\alpha   |u_1 -u_2 | \\ &  \le  C   \eta   \langle x\rangle ^\Imqd    ( |u_1| +  |u_2|)^\alpha   |u_1 -u_2 | 
\\ &  \le  C   \eta    \| u_1 \|_\Spa^\alpha  \|u_1 -u_2 \|_\Spa 
\\ &  =  C    (\eta    \|u_1\|_\Spa )   \|u_1\|_\Spa^{\alpha -1}   \|u_1 -u_2 \|_\Spa  
\\ &  \le   C    (\eta    \|u_1\|_\Spa )   \|u_2\|_\Spa^{\alpha -1}   \|u_1 -u_2 \|_\Spa  
\end{split} 
\end{equation*} 
which yields again a control by the right-hand side of~\eqref{eNL1:2b2}.
We now examine the terms coming from the expression~\eqref{eNL1:4}. We note that~\eqref{eNL1:4} is equal to $ |u|^{\alpha -2p}$ multiplied by a multilinear expression of $u$. Therefore, the difference between the expressions for $u_1$ and $u_2$ can then be written as
\begin{equation}  \label{fSPb2} 
 ( |u_1|^{\alpha -2p} -  |u_2|^{\alpha -2p})  D^\rho u_2 \prod  _{ j=1 }^p D^{\gamma _{1, j}} u_2 D^{\gamma _{2, j}} \overline{u_2}  
\end{equation} 
 plus a sum of terms of the form 
\begin{equation}  \label{fSPb3} 
  |u_1 |^{\alpha -2p} D^\rho w  \prod  _{ j=1 }^p D^{\gamma _{1, j}} w_{1,j} D^{\gamma _{2, j}} \overline{w_{2,j}} 
\end{equation}  
where $w$, $w _{ 1,j }$, $w _{ 2,j }$ are all equal to either $u_1$ or $u_2$, except one of them which is equal to $u_1-u_2$. The terms~\eqref{fSPb3}  can easily be estimated as in Step~1 (Cases~2 and~3), and are controlled by the right-hand side of~\eqref{eNL1:2b2}.
Finally, it remains to estimate the term~\eqref{fSPb2}. We have (using again~\eqref{fSPb1})
\begin{equation*} 
\begin{split} 
| \,  |u_1|^{\alpha -2p} -  |u_2|^{\alpha -2p} | & \le C ( |u_1 |^{-2p-1 } + |u_2 |^{-2p-1 })  ( |u_1| +  |u_2|)^\alpha   |u_1 -u_2 |
 \\ & \le  C   \eta  ^{2p + 1 }   \langle x\rangle ^{ (2p + 1 )\Imqd}  ( |u_1| +  |u_2|)^\alpha  |u_1 -u_2 | 
 \\ & \le  C   \eta  ^{2p + 1 }   \langle x\rangle ^{ (2p -\alpha ) \Imqd }  \| u_1 \|_\Spa ^\alpha  \|u_1 -u_2 \|_\Spa  \\ & =  C   (\eta  \|u_1\|_\Spa )^{2p + 1 }  \langle x\rangle ^{ (2p  -\alpha ) \Imqd}   \|u_1\|_\Spa ^{\alpha -2p-1}  \|u_1 -u_2 \|_\Spa  .
\end{split} 
\end{equation*} 
We can use this estimate (along with~\eqref{fSPb1}) in exactly the same way as we used estimate~\eqref{eNL1:6b1}, and we can conclude as in Step~1. This completes the proof. 
\end{proof} 

\section{Proofs of Theorems~$\ref{eThm1b1}$ and~$\ref{eThm2}$} \label{sNLS} 

We will prove the following result, slightly more general than Theorem~\ref{eThm1b1}.

\begin{thm} \label{eThm1} 
Let $\alpha >0$ and $\lambda \in \C$. Assume~\eqref{fDInt1}-\eqref{fSpa1b1}, let $\Spa$ be defined by~\eqref{fSpa1}-\eqref{fSpa2} and $\Sig$ by~\eqref{fSpa1:b3}.  
Suppose $\DI = e^{i\frac {b |x|^2} {4}} \DIb $ where $b\in \R$, and $\DIb \in \Spa$ satisfies~\eqref{eThm1:1}.
 It follows that there exist $T>0$ and a unique solution $u\in C([-T,T], \Sig ) \cap L^\infty ((-T,T) \times \R^N )$  of~\eqref{NLS1:i}.  Moreover,  the map $t\mapsto  e^{- i \frac {b  |x|^2} {4 (1+ bt)}} u(t,x)$ is continuous $[-T, T] \to \Spa$.
\end{thm} 

We let $b\in \R$ and we consider equation~\eqref{NLS2}, or its equivalent integral form
\begin{equation} \label{fNLS3} 
v(t)= e^{it \Delta } \DIb + i \lambda \int _0^t  (1- b s)^{-\frac {4 - N\alpha } {2}} e^{i(t-s) \Delta }  |v|^\alpha v \, ds .
\end{equation} 
We prove existence results for~\eqref{fNLS3}, of which Theorems~\ref{eThm1} and~\ref{eThm2} are immediate consequences, by using the pseudo-conformal transformation. 

\begin{prop} \label{eNLS1} 
Let $\alpha > 0$ and $\lambda \in \C$. Assume~\eqref{fDInt1}-\eqref{fSpa1b1},  and let the space $\Spa$ be defined by~\eqref{fSpa1}-\eqref{fSpa2}.
 Given any $b\in \R$ and $\DIb \in \Spa$ satisfying~\eqref{eThm1:1}, there exist $0<T<\frac {1} { |b|}$ and a solution $v \in C( [-T,T], \Spa )$ of~\eqref{fNLS3}.
\end{prop} 

\begin{proof} 
We use a standard contraction mapping argument, based on the linear estimates of Proposition~\ref{eLE3} and the nonlinear estimates of Proposition~\ref{eNL1}. We  let 
\begin{equation*} 
\eta >0, \quad K>0, \quad 0<T< \max \{  {\textstyle  \frac {1} { |b|}  } , 1\}
\end{equation*} 
and we define the set $\Ens $ by
\begin{equation} 
\begin{split} 
\Ens = \{ v\in C([-T, T ], \Spa ); & \, \| v \| _{ L^\infty (-T,T), \Spa ) } \le K \text{ and } \\  \eta \langle x \rangle ^\Imqd   |v(t,x)| \ge 1 & \text{ for }  -T<t<T, x\in \R^N \}.
\end{split} 
\end{equation} 
It follows that $\Ens$ equipped with the distance $\dist (u,v)=  \|u-v\|_{L^\infty ((-T,T), \Spa )}$ is a complete metric space. Given  $v\in \Ens$ and $\DIb \in \Spa$, we set
\begin{gather} 
\Phi  _{  v } (t)  =  i \lambda \int _0^t  (1- b s)^{-\frac {4 - N\alpha } {2}} e^{i(t-s) \Delta }  |v|^\alpha v \, ds
\label{fMap1}  \\
\Psi   _{ \DIb, v } (t) = e^{it \Delta } \DIb +  \Phi  _{  v } (t) \label{fMap2}
\end{gather} 
for $-T<t<T$. 
We observe that the definition of $\Ens$ together with Proposition~\ref{eNL1} imply that if $u\in \Ens$, then $ |u|^\alpha u\in C([-T, T], \Ens )$ and
\begin{equation} \label{fSp1b} 
\| \,  |v|^\alpha v \| _{ L^\infty (-T,T), \Spa ) } \le  C  (1+ \eta K)^{2 \Imqh}  K^{\alpha +1} .
\end{equation} 
Using  Proposition~\ref{eLE3}, we deduce that the map $s\mapsto e^{is\Delta } |u(s)|^\alpha u(s)$ belongs to $C([-T,T], \Spa )$, so that (still using Proposition~\ref{eLE3}) $\Phi  _{ v }\in C([-T,T], \Spa )$.
In addition, we deduce from~\eqref{fSp1b} and~\eqref{eLE3:11} that
\begin{equation} \label{fNLS4} 
\| \Phi  _{  v } \| _{ L^\infty (-T, T), \Spa ) } \le
C  |\lambda | f(T) 2^{\Imqh}   
 (1+ \eta K)^{2 \Imqh}  K^{\alpha +1}
\end{equation} 
and
\begin{equation} \label{fNLS5} 
\| \Psi  _{\DIb,  v } \| _{ L^\infty (-T,T), \Spa ) } \le
C  2^{\Imqh}   \bigl(  \|\DIb \|_\Spa +   |\lambda | f(T)  
 (1+ \eta K)^{2 \Imqh}  K^{\alpha +1} \bigr)
\end{equation} 
where
\begin{equation} \label{fNLS5:0} 
f(T)= \int _0^T \max \{ (1- b s)^{-\frac {4 - N\alpha } {2}} , (1+ b s)^{-\frac {4 - N\alpha } {2}}  \}\, ds \goto _{ T\downarrow 0 }0.
\end{equation} 
By a similar argument, it follows from~\eqref{eNL1:2b2} that if $v,w\in \Ens$, then
\begin{equation} \label{fNLS6} 
\| \Phi  _{  v } - \Phi _w \| _{ L^\infty (-T, T), \Spa ) }   \le
C  |\lambda |   f(T)   2^{\Imqh}   
 (1+ \eta K)^{2 \Imqh +1} K^\alpha    \dist (v, w) .
\end{equation} 
Next, we deduce from~\eqref{eLE3:12} and~\eqref{fNLS4}  that
\begin{equation} \label{fNLS7} 
\begin{split} 
 \langle  x \rangle ^\Imqd  |\Psi  _{ \DIb, v }(t,x)|  \ge & \inf _{ x\in \R^N  } \langle x \rangle ^\Imqd  |\DIb (x)| - C T 2^{\Imqh }   \| \DIb  \|_\Spa -  \| \Phi _v (t) \|_\Spa \\  \ge
 \inf _{ x\in \R^N  } \langle x  \rangle ^\Imqd & |\DIb (x)|  - C T   2^{\Imqh}  (  \| \DIb   \|_\Spa +   |\lambda |  f(T) (1+ \eta K)^{2 \Imqh}  K^{\alpha +1} )
\end{split} 
\end{equation} 
for all $-T\le t\le T$ and $x\in \R^N $. 
We now argue as follows.
We denote by $ \widetilde{C} $ the supremum of the constants $C$ in~\eqref{fNLS4}--\eqref{fNLS7} 
and we consider $\DIb \in \Spa$ such that $\inf  _{ x\in \R^N  }  \langle  x \rangle ^\Imqd  |\DIb (x) | >0$.
(Note that there exist such $\DIb$, see Remark~\ref{eRem1}~\eqref{eRem1:3}.) We set
\begin{align} 
\eta & = 2 ( \inf  _{ x\in \R^N  }  \langle  x \rangle ^\Imqd  |\DIb (x) | )^{-1} \label{fNLS8}  \\
K &= 2  \widetilde{C}  2^{\Imqh}   \| \DIb \|_\Spa . \label{fNLS9} 
\end{align} 
It follows in particular that $v(t) \equiv \DIb$ belongs to $\Ens$, so that $\Ens \not = \emptyset$.
We fix $T$ sufficiently small so that 
\begin{gather} 
  \widetilde{C}   2^{\Imqh}     |\lambda | f(T)  
 (1+ \eta K)^{2 \Imqh +1}  K^{\alpha } \le \frac {1} {2} \label{fNLS10} \\
 \widetilde{C}  T   2^{\Imqh}  (  \| \DIb   \|_\Spa +   |\lambda |  f(T) (1+ \eta K)^{2 \Imqh}  K^{\alpha +1} ) \le \frac {1} {\eta} . \label{fNLS11}  
 \end{gather} 
Inequalities~\eqref{fNLS5}, \eqref{fNLS9} and~\eqref{fNLS10} imply that $ \|\Psi  _{ \DIb, v }\| _{ L^\infty (-T, T), \Spa ) } \le K$. Moreover, \eqref{fNLS7}, \eqref{fNLS8} and~\eqref{fNLS11} imply that $ \eta \langle x \rangle ^\Imqd  |\Psi  _{ \DIb, v }(t,x)| \ge 1$ for $-T \le t\le T$ and $x\in \R^N $. 
Thus we see that $\Psi  _{ \DIb, v } \in  \Ens$ for all $v\in \Ens$. 
In addition, it follows from~\eqref{fNLS6} and~\eqref{fNLS10} that the map $v\mapsto \Psi  _{ \DIb, v }$ is a strict contraction $\Ens \to \Ens$. Therefore, it has a fixed point, which is a solution of~\eqref{fNLS3} on $[-T, T]$. 
\end{proof}  

\begin{proof} [Proof of Theorem~$\ref{eThm1}$]
Given $\DIb \in \Spa$ satisfying~\eqref{eThm1:1} and  $b\in \R$, let $0<T < \frac {1} { |b|}$ and $v\in  C([-T, T], \Spa) $ be the solution of~\eqref{fNLS3} given by Proposition~\ref{eNLS1}.
Let  $u$ be defined by
\begin{equation} \label{fNLS1:0} 
u(t, x) = (1+bt )^{-\frac {N} {2}}  e^{i \frac {b  |x|^2} {4 (1+ bt)}} v  \Bigl( \frac {t} {1+bt}, \frac {x} {1+bt} \Bigr)
\end{equation} 
for $-\frac {T} {1+bT}\le t\le \frac {T} {1+bT}$ and $x\in \R^N $. 
Since $v\in C([-T,T], \Sig ) \cap L^\infty ((-T,T)\times \R^N )$, it follows from elementary calculations that 
\begin{equation*} 
u\in C([-\frac {T} {1+bT}, \frac {T} {1+bT}], \Sig ) \cap L^\infty ((-\frac {T} {1+bT}, \frac {T} {1+bT})\times \R^N ) 
\end{equation*} 
is a solution of~\eqref{NLS1:i} on $[-\frac {T} {1+bT}, \frac {T} {1+bT}]$ and $u(0) = e^{i\frac {b |x|^2} {4}} \DIb $. 
See e.g.~\cite[Section~3]{CWrapdec}. 
Finally, uniqueness in  $C([-T,T], \Sig ) \cap L^\infty ((-T,T) \times \R^N )$ is immediate (because of the $L^\infty $ bound), and the proof is complete.
\end{proof} 

\begin{prop} \label{eNLS1b1} 
Let $\alpha >\frac {2} {N}$ and $\lambda \in \C$. Assume~\eqref{fDInt1}-\eqref{fSpa1b1},  and let the space $\Spa$ be defined by~\eqref{fSpa1}-\eqref{fSpa2}.
 Given any $\DIb \in \Spa$ satisfying~\eqref{eThm1:1}, there exists  a solution $v \in C( [0, \frac {1} {b}], \Spa )$ of~\eqref{fNLS3} on $[0, \frac {1} {b}]$ provided $b>0$ is sufficiently large. 
\end{prop} 

\begin{proof} 
The proof is similar to the proof of Proposition~\ref{eNLS1}.
The difference is that, instead of assuming $T$ small, we now make $b>0$ large in order to apply the contraction principle. 
We  let $b, \eta ,K>0$ and we define the set $\Ens $ by
\begin{equation} 
\begin{split} 
\Ens = \{ v\in C([0, {\textstyle {\frac {1} {b}}}  ], \Spa ); & \, \| v \| _{ L^\infty (0, \frac {1} {b}), \Spa ) } \le K \text{ and } \\  \eta \langle x \rangle ^\Imqd   |v(t,x)| \ge 1 & \text{ for }  0\le t\le{\textstyle  \frac {1} {b}}, x\in \R^N \}
\end{split} 
\end{equation} 
so that $\Ens$ equipped with the distance $\dist (u, v) =  \| u-v \|_{L^\infty ((0, \frac {1} {b}), \Spa )}$ is a complete metric space. 
Given  $v\in \Ens$ and $\DIb \in \Spa$, we consider $\Phi  _{  v } (t)$ and $\Psi   _{ \DIb, v } (t)$ defined by~\eqref{fMap1} and~\eqref{fMap2} for $0\le  t < \frac {1} {b}$. 
Arguing as in the proof of Proposition~\ref{eNLS1}, we see that
\begin{equation} \label{fSp1b:0} 
\| \,  |v|^\alpha v \| _{ L^\infty (0, { {\frac {1} {b}}}), \Spa ) } \le  C  (1+ \eta K)^{2 \Imqh}  K^{\alpha +1} 
\end{equation} 
 that  $\Phi  _{ v }\in C([0, \frac {1} {b}], \Spa )$ and that
\begin{gather} 
\| \Phi  _{  v } \| _{ L^\infty (0,  \frac {1} {b}), \Spa ) } \le C  |\lambda | {\textstyle  \frac {1} {b}   (  1+\frac {1} {b}  )^{\Imqh}   }  (1+ \eta K)^{2 \Imqh}  K^{\alpha +1}  \label{fNLS4b1} \\
\| \Psi  _{\DIb,  v } \| _{ L^\infty (0, { {\frac {1} {b}}}), \Spa ) } \le
C {\textstyle (  1+\frac {1} {b}  )^{\Imqh}   \bigl(  \|\DIb \|_\Spa +   |\lambda | \frac {1} {b}  
 (1+ \eta K)^{2 \Imqh}  K^{\alpha +1} \bigr)}  \label{fNLS5b1}  \\
\| \Phi  _{  v } - \Phi _w \| _{ L^\infty (0, { {\frac {1} {b}}}), \Spa ) }  \le
C  |\lambda |  {\textstyle \frac {1} {b}   (  1+\frac {1} {b}  )^{\Imqh}   }
 (1+ \eta K)^{2 \Imqh +1} K^\alpha    \dist (v, w). \label{fNLS6b1} 
\end{gather} 
(In the last three inequalities, we use the identity $\int _0^{\frac {1} {b}} (1- b s)^{-\frac {4 - N\alpha } {2}} ds = \frac {2} {(N\alpha -2)b}  $.)
Next, it follows from~\eqref{eLE3:12} and~\eqref{fNLS4b1}  that
\begin{equation} \label{fNLS7b1} 
\begin{split} 
 \langle  x \rangle ^\Imqd  |\Psi  _{ \DIb, v }(t,x)|  \ge & \inf _{ x\in \R^N  } \langle x \rangle ^\Imqd  |\DIb (x)| - C t (1+t)^{\Imqh }   \| \DIb  \|_\Spa -  \| \Phi _v (t) \|_\Spa \\  \ge 
 \inf _{ x\in \R^N  } \langle x \rangle ^\Imqd &  |\DIb (x)|  - C  {\textstyle  \frac {1} {b}   (  1+\frac {1} {b}  )^{\Imqh}}  (  \| \DIb   \|_\Spa +   |\lambda |  (1+ \eta K)^{2 \Imqh}  K^{\alpha +1} ).
\end{split} 
\end{equation} 
We now argue as follows.
We denote by $ \widetilde{C} $ the supremum of the constants $C$ in~\eqref{fNLS4b1}--\eqref{fNLS7b1}.  
We consider $\DIb \in \Spa$ such that $\inf  _{ x\in \R^N  }  \langle  x \rangle ^\Imqd  |\DIb (x) | >0$.
(Such $\DIb$ exist, , see Remark~\ref{eRem1}~\eqref{eRem1:3}), we let $\eta$ be defined by~\eqref{fNLS8} and we set
\begin{equation} \label{fNLS9b1} 
K = 4  \widetilde{C}  \| \DIb \|_\Spa. 
\end{equation} 
It follows in particular that $v(t) \equiv \DIb$ belongs to $\Ens$, so that $\Ens \not = \emptyset$.
We consider $b>0$ sufficiently large so that 
\begin{gather} 
 {\textstyle (  1+\frac {1} {b}  )^{\Imqh } } \le 2 \label{fNLS10b1} \\
 |\lambda | {\textstyle \frac {1} {b} }    (1+ \eta K ) ^{2\Imqh}    K^{\alpha +1} \le  \| \DIb \|_\Spa \label{fNLS11b1}  \\ 2\eta  \widetilde{C}   {\textstyle  \frac {1} {b}  }  (  \| \DIb   \|_\Spa +  |\lambda |  (1+  \eta K ) ^{2\Imqh}    K^{\alpha +1} ) \le 1 \label{fNLS12b1} \\
4 \widetilde{C}   |\lambda |  {\textstyle \frac {1} {b}  }
 (1+ \eta K)^{2 \Imqh +1} K^\alpha    \le 1 . \label{fNLS13b1} 
 \end{gather} 
Inequalities~\eqref{fNLS5b1}, \eqref{fNLS10b1}, \eqref{fNLS11b1} and~\eqref{fNLS9b1} imply that $ \|\Psi  _{ \DIb, v }\| _{ L^\infty (0, \frac {1} {b}), \Spa ) } \le K$. Moreover, \eqref{fNLS7b1}, \eqref{fNLS8}, \eqref{fNLS10b1} and~\eqref{fNLS12b1} imply that $ \eta \langle x \rangle ^\Imqd  |\Psi  _{ \DIb, v }(t,x)| \ge 1$ for $0\le t\le \frac {1} {b}$ and $x\in \R^N $. 
Thus we see that $\Psi  _{ \DIb, v } \in  \Ens$ for all $v\in \Ens$. 
In addition, it follows from~\eqref{fNLS6b1}, \eqref{fNLS10b1} and~\eqref{fNLS13b1} that the map $v\mapsto \Psi  _{ \DIb, v }$ is a strict contraction $\Ens \to \Ens$. Therefore, it has a fixed point, which is a solution of~\eqref{fNLS3} on $[0, \frac {1} {b}]$. 
\end{proof} 

\begin{rem} 
One can solve under the same conditions the problem with final value $\psi $ at time $\frac {1} {b}$, i.e. 
\begin{equation} \label{fNLS3b1} 
v(t)= e^{i(t -\frac {1} {b}) \Delta } \psi  + i \lambda \int _t^{\frac {1} {b}}  (1- b s)^{-\frac {4 - N\alpha } {2}} e^{i(t-s) \Delta }  |v|^\alpha v \, ds
\end{equation} 
\end{rem} 

\begin{proof} [Proof of Theorem~$\ref{eThm2}$]

Given $\DIb \in \Spa$ satisfying~\eqref{eThm1:1}, let $b>0$ be sufficiently large so that there exists a solution $v\in  C([0, \frac {1} {b}], \Spa) $ of~\eqref{fNLS3}, by Proposition~\ref{eNLS1b1}.
Let  $u$ be defined by~\eqref{fNLS1:0} for $0\le t<\infty $ and $x\in \R^N $. 
It follows from elementary calculations that 
\begin{equation*} 
u\in C([0, \infty ), \Sig ) \cap L^\infty ((0,\infty )\times \R^N ) 
\end{equation*} 
is a solution of~\eqref{NLS1:i} on $[0, \infty )$. 
Moreover, $u(0) = e^{i\frac {b |x|^2} {4}} \DIb $, and 
$e^{-it\Delta } u(t) \to u^+$ in $\Sig$ as $t\to \infty $, where $u^+= e^{i\frac {b |x|^2} {4}} e^{-\frac {i} {b} \Delta } v(\frac {1} {b})$. See e.g.~\cite[Proposition~3.14]{CWrapdec}. 
In addition, since $\sup  _{ 0\le t\le 1/b}  \| v(t)\| _{ L^\infty  }< \infty $, it follows from~\eqref{fNLS1:0} that   $\sup  _{ t\ge 0 } (1+t)^{\frac {N} {2}}  \| u(t) \| _{ L^\infty  }<\infty $.
Finally, uniqueness in  $C([0, \frac {1} {b}], \Sig ) \cap L^\infty ((0, \frac {1} {b}) \times \R^N )$ is immediate (because of the $L^\infty $ bound), and the proof is complete.
\end{proof} 

\begin{rem} \label{eRem3} 
Note that the solution $u$ of~\eqref{NLS1:i} in Theorem~\ref{eThm2} is obtained by applying the pseudo-conformal transformation to the solution $v$ of~\eqref{fNLS3} constructed in Proposition~\ref{eNLS1b1}. It follows that it has stronger regularity properties than stated in Theorem~\ref{eThm2}. Indeed, it follows easily from formula~\eqref{fNLS1:0} that the map $t\mapsto  e^{- i \frac {b  |x|^2} {4 (1+ bt)}} u(t,x)$ is continuous  $[0,\infty ) \to \Spa$.
In addition, the scattering state $u^+$ satisfies $e^{- i \frac {b  |x|^2} {4 }} u^+ \in \Spa$.
\end{rem} 

\appendix 
\section{Some elementary estimates} \label{sElem} 

In this section, we collect a point-wise estimate (Lemma~\ref{eA1}) and an interpolation estimate (Lemma~\ref{eA2}) which we use in this paper.
For the proof of Lemma~\ref{eA1}, we will use the following observation.

\begin{rem} \label{eRemDr} 
Let $f:\R \to \R$ and $g: \R^N \to \R $. Given a multi-index $\alpha $ with $ |\alpha |\ge 1$, $D^\alpha  f( g(x) )$ is a sum of terms of the form $f^{(\Inzu)} ( g(x) ) \prod  _{ \Inzd =1 }^\Inzu D^{\beta _\Inzd } g(x) $ where $1\le \Inzu \le  |\alpha |$, and the $\beta _\Inzd $ are multi-indices such that $ |\beta _\Inzd |\ge 1$ and $ |\beta _1|+  \cdots +  |\beta _\Inzu | =  |\alpha |$.
\end{rem} 

\begin{lem} \label{eA1} 
Given $\eta \in \R$ and a multi-index $\alpha $, there exists a constant $C$ such that
\begin{equation}  \label{fSob5} 
 | D^\alpha (\langle x\rangle ^\eta u) | \le C \sum_{ \Inzt =0 }^{ |\alpha |}  \langle x\rangle ^{ \eta -  |\alpha |+\Inzt } \sum_{  |\beta |=\Inzt }  |D^\beta u|
\end{equation} 
for all $u\in C^{ |\alpha |} (\R^N , \C)$.
\end{lem} 

\begin{proof} 
We first claim that
\begin{equation}  \label{fSob3} 
 | D^\alpha \langle x\rangle  | \le C_\alpha \langle x\rangle ^{1-  |\alpha |} .
\end{equation} 
Indeed, we apply Remark~\ref{eRemDr} with $f(t)= (1+t)^{\frac {1} {2}}$ and $g(x)=  |x|^2$, so that $\langle x\rangle= f(g(x))$.
We note that
\begin{equation*} 
f^{(\Inzq)} (t)= c_\Inzq (1+t)^{\frac {1} {2}- \Inzq}
\end{equation*} 
and that 
$\partial _j g= 2x_j$, $\partial  _{ i j }g= 2\delta  _{ ij }$, $D^\beta g=0$ if $ |\beta |\ge 3 $.
In particular, 
\begin{equation*} 
 | D^\beta  g(x)| \le c_\alpha  |x|^{2-  |\beta |}
\end{equation*} 
so that the generic term in the development of $D^\alpha \langle x\rangle $ given by Remark~\ref{eRemDr} can be estimated by $\langle x\rangle ^{1-2\Inzq} \prod  _{ \Intd =1 }^\Inzq  |x|^{2-  |\beta _\Intd |} $.
Since $ |\beta _1|+  \cdots +  |\beta _\Intd| =  |\alpha |$, we obtain~\eqref{fSob3}. 
We now consider a real number $\eta $ and we apply Remark~\ref{eRemDr} with $f(t)= t^\eta $ and $g(x)=  \langle x\rangle $. Using~\eqref{fSob3}, we deduce easily that
\begin{equation}  \label{fSob4} 
 | D^\alpha \langle x\rangle ^\eta | \le C_{\eta, \alpha }  \langle x\rangle ^{\eta -  |\alpha |}
\end{equation} 
and~\eqref{fSob5} follows by applying Leibniz's rule.
\end{proof} 

\begin{lem} \label{eA2} 
Given $\Inzc \in \N$ and $\Inth \in \R$, there exists a constant $C$ such that
\begin{equation} \label{fSob6} 
 \sup  _{  |\beta  |= \Inzc +1 }\| \langle x\rangle ^\Inth D^\beta  u \| _{ L^\infty  } \le C ( \sup  _{  |\beta  |= \Inzc  }\| \langle x\rangle ^\Inth D^\beta  u \| _{ L^\infty  } +  \sup  _{  |\beta  |= \Inzc +2 }\| \langle x\rangle ^\Inth D^\beta  u \| _{ L^\infty  })
\end{equation} 
for all  $u \in  C^{j+2} (\R^N )$.
\end{lem} 

\begin{proof} 
It suffices to show that
\begin{equation}  \label{fSob7}
\| \langle x\rangle ^\Inth \nabla u \| _{ L^\infty  } \le C ( \| \langle x\rangle ^\Inth  u \| _{ L^\infty  } +  \sup  _{  |\alpha |= 2 }\| \langle x\rangle ^\Inth D^\alpha u \| _{ L^\infty  }) .
\end{equation} 
Given $x, y\in \R^N $, we have 
\begin{equation*}
 u(y) -   u (x) =  \int _0^1 \frac {d} {ds} u (x+ s(y-x)) \, ds= \int _0^1 (y-x)\cdot \nabla u  (x+ s(y-x))\, ds
\end{equation*} 
and
\begin{equation*}
\begin{split} 
 \nabla u  (x+ s(y-x)) - \nabla u  (x)  & =  \int _0^1 \frac {d} {d\sigma } \nabla u  (x+ \sigma s(y-x))\, d\sigma  \\
 & = \int _0^1  s(y-x) \cdot \nabla ^2 u  (x+ \sigma s(y-x))\, d\sigma 
\end{split} 
\end{equation*} 
which imply
\begin{equation*}
\begin{split} 
u(y)- u(x)=&  \int _0^1 (y-x)\cdot \nabla u  (x )\, ds \\ &+ \int _0^1 \int _0^1 (y-x)\cdot [ s(y-x) \cdot \nabla ^2 u  (x+ \sigma s(y-x)) ]\, d\sigma ds .
\end{split} 
\end{equation*} 
If $\nabla u  (x) \not = 0$, we let $y= x+ \frac {\nabla  \overline{u}   (x)} { |\nabla u  (x) |}$, and we obtain
\begin{equation*}
\begin{split} 
u(y)- u(x)=&   |\nabla u  (x ) | \\ &+ \int _0^1 \int _0^1 (y-x)\cdot [ s(y-x) \cdot \nabla ^2 u  (x+ \sigma s(y-x)) ]\, d\sigma ds
\end{split} 
\end{equation*} 
so that
\begin{equation*}
  |\nabla u  (x ) |  \le  2 \sup  _{  |y-x|\le 1 }  |u(y)| +  \sup  _{  |y-x|\le 1 }\sup  _{  |\beta |=2 }  |D^\beta u(y)| .
\end{equation*} 
Estimate~\eqref{fSob7} easily follows.
\end{proof}

\end{document}